\documentclass[11pt]{article}
\usepackage{amsmath,amsfonts,amssymb,amscd}
\usepackage{graphicx}

\newtheorem{theo}{Theorem}
\newtheorem{lem}{Lemma}[section]
\newtheorem{prop}{Proposition}[section]

\newtheorem{cor}{Corollary}[section]

\newtheorem{dfn}{Definition}[section]

\makeatletter \@addtoreset{equation}{section} \makeatother

\newcommand{\mC}{\mathbb{C}}

\newcommand{\mR}{\mathbb{R}}

\newcommand{\mN}{\mathbb{N}}

\newcommand{\calC}{{\cal C}}

\newcommand{\calF}{{\cal F}}

\newcommand{\calH}{{\cal H}}

\newcommand{\eps}{\varepsilon}
\newcommand{\ph}{\varphi}
\newcommand{\thet}{\vartheta}

\newcommand{\id}{\operatorname{id}}

\newcommand\qed{{\unskip\nobreak\hfil\penalty50
  \hskip2em\hbox{}\nobreak\hfil\mbox{\rule{1ex}{1ex} \qquad}
    \parfillskip=0pt \finalhyphendemerits=0\par\medskip}}

\begin{document}

\title
{Isonormal potentials in one degree of freedom}
\author{ D.~Treschev\\
 Steklov Mathematical Institute of Russian Academy of Sciences
\date{}
}

\maketitle

\begin{abstract}
We consider the Hamiltonian system with one degree of freedom
$$
  \dot x = y, \quad \dot y = -\partial V(x) / \partial x, \qquad x,y\in\mR
$$
with the smooth potential $V:\mR\to\mR$. We assume that the origin is an elliptic singular point: $V'(0)=0$ and $V''(0)>0$. Two potentials $V_0$ and $V_1$ are refered to be isonormal if in a neighborhood of the origin there exists a canonical change of coordinates which transforms the Hamiltonian function $y^2/2 + V_0$ to $y^2/2 + V_1$. In particular, any potential isonormal to $x^2/2$ is isochronous. We obtain a criterium of isonormality for analytic potentials.
\end{abstract}

\section{Introduction}

Let $V:\mR\to\mR$ be a $C^2$-smooth function. Consider the ODE system
\begin{equation}
\label{newton}
  \ddot x = -\partial_x V, \qquad  x\in\mR.
\end{equation}
The function $V$ is called a potential. It is assumed to have $0$ as a point of nondegenerate minimum i.e.,
$V'(0) = 0$, $V''(0) > 0$. The first integral
$$
  y^2 / 2 + V(x) = \mbox{constant}, \quad   y = \dot x.
$$
implies that all solutions of (\ref{newton}) with initial conditions near the origin $(x,y)=(0,0)$ are periodic. If the periods of all these solutions are the same, the system (\ref{newton}) as well as the potential $V$ are called isochronous.

Complete characterization of local isochronous potentials was obtained by efforts of several authors:
\cite{KP, U61, Str, GZ}. Many results (including multidimensional ones) on the subject are contained in the book \cite{Calogero_book}. A series of global analytic examples was obtained in \cite{StSt, BM2003, Dor, ACMP, Mam}.

In \cite{Tre_RCD22, Tre_Proc23} isochronicity problem in a more general context was considered: instead of (\ref{newton}) analytic Hamiltonian systems with one degree of freedom and Hamiltonian
$$
  y^2/2 + x^2/2 + \mbox{higher order terms}
$$
were studied. A criterium of isochronicity was obtained (without a complete proof). The proof is given in \cite{Tre26}.

In \cite{U62, Tre22} an extended version of the problem is proposed: given a function which connects the period of the oscillations with the energy\footnote{equivalently, given a Hamiltonian normal form for a system with one degree of freedom}, find a Hamiltonian (in a certain class of Hamiltonians) which realizes such a connection\footnote{respectively, which has this normal form}. In particular, if this function is a constant\footnote{respectively, if the normal form is trivial}, we have the isochronicity problem. In \cite{Tre_RMS}
we propose an explicit formula which computes normal form for a one degree of freedom Hamiltonian without any normalization procedure. This formula is proven in \cite{Tre26}, where a multidimensional generalizations are also presented.

In this paper we obtain a necessary and sufficient condition for two analytic potentials $V$ in equation (\ref{newton}) to generate the same dependence of the oscillation period on the energy. In particular we reprove the result from \cite{GZ} on the complete description of isochronous potentials.

We study Hamiltonian systems
\begin{equation}
\label{ini}
  \dot x = \partial_y H, \quad
  \dot y = - \partial_x H
\end{equation}
with one degree of freedom on the symplectic plane\footnote
{The problems we consider are local. So, our phase space is in fact a small neighborhood of the origin in $\mR^2$.}
$\mR^2=\{(x,y)\}, dy\wedge dx$ and Hamiltonian
\begin{equation}
\label{ham}
  H = y^2/2 + V(x), \qquad
  V(x) = v^2(x) / 2, \quad v(0) = 0, \quad v'(0) = 1.
\end{equation}
This system is equivalent to (\ref{newton}). Hamiltonian formalism will be important for our argument.
Below all functions and transformations are assumed to be Taylor series at the origin with positive radius of convergence.

\begin{dfn}
Two potentials $V_0$ and $V_1$ will be said to be isonormal if there exists a real-analytic symplectic coordinate change $(B\subset\mR^2,0)\to (\mR^2,0)$ (here $B$ is a neighborhood of the origin) which transforms one Hamiltonian (\ref{ham}) to another. In the other words, potentials are isonormal if the Birkhoff normal forms at the elliptic singular point $0\in\mR^2$ for the corresponding systems coincide.
\end{dfn}

\begin{theo}
\label{theo:iso}
The potentials $V_0 = v_0^2 / 2$ and $V_1 = v_1^2 / 2$ are isonormal if and only if odd parts of the functions\footnote
{$v^{-1}$ denotes the function inverse to $v$ i.e., $v\circ v^{-1}(x) \equiv v^{-1} \circ v(x) \equiv x$.}
$v_1^{-1}$ and $v_2^{-1}$ coincide:
\begin{equation}
\label{vvvv}
  v_0^{-1}(z) - v_0^{-1}(-z) = v_1^{-1}(z) - v_1^{-1}(-z).
\end{equation}
\end{theo}

\begin{cor}
Any potential $V = v^2/2$  is isonormal to the even potential $V_e = v_e^2 / 2$, where $v_e$ is the odd function, determined by the equation
$$
  v_e^{-1}(z) = \frac12 \big( v^{-1}(z) - v^{-1}(-z) \big).
$$
\end{cor}

Below we prove (Proposition \ref{prop:uni}) that any analytic in a neighborhood of $0\in\mC$ function $v$ such that $v(0)=0$ and $v'(0)\ne 0$ is uniquely presented in the form $v(x) = \psi^{-1} (x - h(x))$, where $\psi = \psi_v$ is odd and $h = h_v$ is an involution (see (\ref{hpsi})). This implies

\begin{cor}
The potentials $V_0 = v_0^2 / 2$ and $V_1 = v_1^2 / 2$ are isonormal if and only if there exists an odd
$\psi$ and two involutions $h_0, h_1$ such that
$$
  v_0(x) = \psi^{-1} (x - h_0(x)) \quad \mbox{and} \quad
  v_1(x) = \psi^{-1} (x - h_1(x)).
$$
\end{cor}

The potential $V = x^2 / 2$ and potentials isonormal to it are isochronous: in the corresponding system (\ref{ini}) period of oscillations near the origin does not depend on the energy. If $V$ is isochronous, (\ref{vvvv}) implies $v^{-1}(z)-v^{-1}(-z)=2z$. Taking $z=v(\zeta)$, we obtain: $\zeta - h(\zeta) = 2v(\zeta)$, where
$h=v^{-1}\circ(-\id)\circ v$ is an involution: $h(0)=0$, $h'(0)=-1$, $h\circ h(z)=z$.

Such an isochronicity criterium, $v(z) = (z - h(z))/2$ with an involution $h$, was obtained by Gorni and Zampieri \cite{GZ} in the $C^2$-smooth category.

Equation (\ref{vvvv}) can be obtained by using quadratures. Indeed, equation (12,1) from \cite{LL} reads
\begin{equation}
\label{quad}
  x_+(V) - x_-(V) = \frac{1}{\sqrt{2} \pi} \int_0^V \frac{T(E)\, dE}{\sqrt{V - E}}.
\end{equation}
Here $T(E)$ determines period of oscillations as a function of the energy $E = \dot x^2 / 2 + V(x)$. The potential
$V = V(x)$ increases on an interval $(0,x_0)$ and decreases on $(-x_0,0)$, $x_0 > 0$. The corresponding inverse functions are denoted $x_\pm = x_\pm(V)$.

Let $v = v(x)$ be the unique smooth function such that $V = v^2/2$, $v|_{(0,x_0)} > 0$, $v|_{(-x_0,0)} < 0$. We also put $E = \chi^2/2$. Then (\ref{quad}) takes the form
$$
  x(v) - x(-v) = \frac{1}{\pi} \int_0^v \frac{T(\chi^2/2) \chi\, d\chi}{\sqrt{v^2 - \chi^2}}, \qquad
  v > 0,
$$
where $x = x(v)$ is the function inverse to $v = v(x)$. This equation implies that the function
$$
  z \mapsto \psi(z) = v^{-1}(z) - v^{-1}(-z)
$$
is completely determined by the function $T=T(E)$. Hence for isonormal potentials the function $\psi$ is the same.

Below we give another proof of Theorem \ref{theo:iso}. Unlike the method of quadratures our approach has a chance to be extended to multidimensional situation.

\section{Transformations preserving form of Hamiltonian (\ref{ham})}

In this section we determine the class of local symplectic transformations of $(\mR^2,0)$ which preserve the form of the Hamiltonian $H$, changing only the potential $V$. We take these transformations as shifts
$(x,y) = (x(0),y(0))\mapsto (x(\delta),y(\delta))$ along solutions of an auxiliary Hamiltonian system
\begin{equation}
\label{hamF}
  dx/d\delta = \partial F / \partial y, \quad
  dy/d\delta = - \partial F / \partial x, \qquad
  F = F(x,y,\delta).
\end{equation}
Then the Hamiltonian $\calH(x(\delta),y(\delta),\delta) = H(x,y)$ satisfies the equations
\begin{equation}
\label{ave}
  \partial_\delta\calH = \{F,\calH\}, \qquad
  \calH|_{\delta=0} = H,
\end{equation}
where $\{\,,\}$ is the Poisson bracket. The Hamiltonian function has the form\footnote{Here for simplicity of the notation we write $x,y$ instead of $x(\delta),y(\delta)$.}
$$
  \calH(x,y,\delta) = y^2/2 + V(x,\delta)
$$

It is convenient to present the potential $V = V(x,\delta)$ in the form
$$
  V = \frac12 v^2(x,\delta), \qquad  v(0,\delta) = 0, \quad  \partial_x v(0,\delta) = 1.
$$
We expand $F$ in the series
\begin{equation}
\label{F}
  F = \sum_{k=0}^\infty  y^k f_k(x,\delta).
\end{equation}
Then system (\ref{ave}) takes the form
$$
    v\partial_\delta v
  = \sum_{k=0}^\infty \Big( k y^{k-1} f_k v\partial_x v - y^{k+1} \partial_x f_k \Big).
$$
Equating in this equation terms at equal degrees of $y$, we obtain:
\begin{eqnarray}
\nonumber
       \partial_\delta v
  &=& f_1 \partial_x v, \qquad\qquad\qquad\,
       0
 \;=\; 2f_2 v\partial_x v - \partial_x f_0, \\
\label{ave+}
      0
  &=& 3f_3 v\partial_x v - \partial_x f_1, \qquad
      0
 \;=\; 4f_4 v\partial_x v - \partial_x f_2, \\
\nonumber
      0
  &=& 5f_5 v\partial_x v - \partial_x f_3, \qquad
      0
 \;=\; 6f_6 v\partial_x v - \partial_x f_4,\\
\nonumber
      \ldots
  &&  \ldots \qquad\qquad \ldots \qquad
      \ldots \qquad \ldots \qquad\qquad \ldots
\end{eqnarray}
The functions $f_k$ with even $k$ may be put equal to zero. The equations $v = O(x)$, $\partial_x v = O(1)$, $\partial_\delta v = O(x^2)$ imply $f_1 = O(x^2)$.

First, we discuss formal solvability of system (\ref{ave+}). We define the differential operator
$$
  D = \frac{1}{v\partial_x v} \partial_x.
$$
Then for any $k\in\mN = \{1,2,\ldots\}$
\begin{equation}
\label{f_2k+1}
  f_{2k+1} = \frac1{2k+1} D f_{2k-1} = \frac1{(2k+1)!!} D^k f_1.
\end{equation}
To avoid singularities at zero for functions $f_{2k+1}$, we have to require $\partial_x|_{x=0} D^k f_1 = 0$. This system of conditions can be written in a more convenient form if $f_1$ is assumed to be a function of $v$ and $\delta$. Then $D = \frac1v \partial_v$ and the conditions become
\begin{equation}
\label{partialf}
  \partial_v \big|_{v=0} \Big(\frac1v \partial_v\Big)^k f_1 = 0, \qquad  k\in\mN
\end{equation}
i.e., $f_1$ (as well as any function $f_{2k+1}$) is even w.r.t. $v$.

The following lemma shows that if $f_1$ is a series with a positive convergence radius then $F$ is a Taylor series,  converging in a neighborhood of the point $0\in\mC^2 = \{(x,y)\}$.

\begin{lem}
\label{lem:conver}
Suppose $f_1 = f_1(v^2)$ satisfies the equation $\max_{|v^2|\le r} |f_1| = M$. Then for any $\eps\in (0,r)$
$$
  \max_{|v^2|\le r-\eps} |f_{2k+1}(v^2)| \le \eps^{-k} M.
$$
\end{lem}

{\it Proof}. Let $S\subset\mC$ be the positively oriented circle with the center at the origin and radius $r$. We put $u=v^2$. By the Cauchy formula
$$
  f_1(u) = \frac{1}{2\pi i} \int_S \frac{f(\xi)\, d\xi}{\xi-u}.
$$
Then
$$
    f_{2k+1}(u)
  = \frac{2^k}{2\pi i (2k+1)!!} \partial_u^k \int_S \frac{f(\xi)\, d\xi}{\xi-u}
  = \frac{(2k)!!}{2\pi i (2k+1)!!} \int_S \frac{f(\xi)\, d\xi}{(\xi-u)^{k+1}}.
$$
This implies that for any $|u|\le r-\eps$
$$
      |f_{2k+1}(u)|
  \le \frac{(2k)!!}{2\pi (2k+1)!!} \int_S \frac{M\, d\xi}{\eps^{k+1}}
  \le \frac{M}{\eps^k}.
$$
\qed

\section{Space of power series and its subspaces}

Let $\calF$ be the space of formal power series
$$
  v(x) = \sum_{j=1}^\infty v_j x^j, \qquad  v_j\in\mC .
$$
with positive (depending on $v$) radius of convergence. Any $v\in\calF$ determines a function holomorphic in a neighborhood of the point $0\in\mC$. We define $\id\in\calF$, $\id(x) = x$.

The space $\calF$ with the operations $+,-,\cdot$, and multiplication by a complex number is a commutative associative algebra over $\calC$. We also have the composition operation $\circ$ (the substitution of a series into a series) and the operation $v\mapsto v^{-1}$ of taking the inverse element (if $v_1\ne 0$) such that $v\circ v^{-1} = v^{-1}\circ v = \id$.

We will need the following subsets in $\calF$.
\begin{itemize}
\item $\calF_*\subset\calF$, series with $v_1\ne 0$. Note that $(\calF_*,\circ)$ is a group with the unit element $\id$.

\item $\calF_{odd}\subset\calF$, the space of odd series. Hence, $\calF_{odd *} = \calF_{odd}\cap \calF_*$ is a subgroup in $\calF_*$.

\item $\calF_{even}\subset\calF$, the space of even series.

\item $\calF_{inv}\subset\calF$, the set of involutions $h = \sum_{k=1}^\infty h_j x^j$, $h_1 = -1$, $h\circ h = \id$.
\end{itemize}

The simplest involution is $-\id$. For any $v\in\calF_*$ we have the involution $h_v \in \calF_{inv}$ and the series $\psi_v\in\calF_{odd}$
\begin{equation}
\label{hpsi}
  h_v = v^{-1}\circ (-\id)\circ v, \quad
  \psi_v = v^{-1} - v^{-1}\circ (-\id).
\end{equation}
Then we have the identity
\begin{equation}
\label{vhpsi}
  \id - h_v = \psi_v\circ v.
\end{equation}

\begin{prop}
\label{prop:uni}
Suppose $v\in\calF_*$, $h\in\calF_{inv}$, and $\psi\in\calF_{odd}$ satisfy the equation
\begin{equation}
\label{uni}
   \id - h = \psi\circ v.
\end{equation}
Then $h=h_v$ and $\psi = \psi_v$.
\end{prop}

We prove this proposition in Section \ref{sec:inv}. According to Proposition \ref{prop:uni} the map $v\mapsto (\psi_v,h_v)$ determines a one-to-one correspondence between $\calF_*$ and
$\calF_{odd *}\times\calF_{inv}$.

\section{Evolution of the potential}

Let the function $\ph = \ph(x,\delta)$ be continuous in $\delta$ and lie in $\calF_{even}$ for any fixed $\delta\in [0,\delta_0]$. Consider the PDE (compare with the first equation (\ref{ave+}))
\begin{equation}
\label{pde}
  \partial_\delta v = \ph(v,\delta) \partial_x v,  \qquad
  v|_{\delta=0} = v_0 \in \calF_*.
\end{equation}

\begin{prop}
\label{prop:pde}
There exists a unique formal solution $v = v(x,\delta)$, $\delta\in [0,\delta_0]$. It lies in $\calF_*$ if
$\delta\in [0,\tilde\delta_0]$ for some $0<\tilde\delta_0\le\delta_0$. Moreover, for any
$\delta\in [0,\tilde\delta_0]$
$$
  \psi_{v(\cdot,\delta)} = \psi_{v_0}.
$$
\end{prop}

{\it Proof}. Existence and uniqueness of a solution for the initial value problem (\ref{pde}) is a standard fact from the theory of first order PDE's (see for example, equation (\ref{implicit}) below).

We put $v = \psi_{v_0}^{-1}\circ u$. Then (\ref{pde}) takes the form
\begin{equation}
\label{pde+}
  \partial_\delta u(x,\delta) = \ph\circ\psi_{v_0}^{-1}\circ u(x,\delta) \partial_x u(x,\delta),  \qquad
  u|_{\delta=0} = \id - h_{v_0}.
\end{equation}
By (\ref{vhpsi}) the function $h_{v_0}$ is an involution. We put $u = \id - h$. Hence we have to prove that for any $\delta\in [0,\tilde\delta_0]$ the function $h(\cdot,\delta) =  \id - u(\cdot,\delta)$  lies in $\calF_{inv}$.

Equation (\ref{pde+}) implies
\begin{equation}
\label{pde++}
  -\partial_\delta h = \Phi(\id - h) \big(1 - \partial_x h\big),  \qquad
  h|_{\delta=0} = h_{v_0},
\end{equation}
where $\Phi = \ph\circ\psi_{v_0}^{-1}\in\calF_{even}$.

We differentiate $h\circ h(\cdot,\delta)$ w.r.t. $\delta$ assuming that $h(\cdot,\delta)\in\calF_{inv}$:
\begin{eqnarray*}
      \partial_\delta (h\circ h)
  &=& (\partial_\delta h)\circ h + (\partial_x h)\circ h \cdot \partial_\delta h  \\
  &=& - \Phi(h - h\circ h)\cdot \big( 1 - (\partial_x)\circ h \big)
      - (\partial_x h)\circ h \cdot \Phi(\id - h) \cdot (1 - \partial_x h)  \\
  &=& \Phi(\id - h)\cdot\Big( -1 + (\partial_x h)\circ h - (\partial_x h)\circ h + \partial_x (h\circ h) \Big)
   = 0.
\end{eqnarray*}
Hence $h(\cdot,\delta)$ remains an involution for all $\delta$. \qed

\begin{cor}
The functional $v\mapsto\psi_v$ is a first integral (a constant of motion) for system (\ref{pde}).
\end{cor}

\begin{cor}
Since the function $v^{-1} - v^{-1}\circ(-\id)$ does not depend on $\delta$ on solutions of system (\ref{ave+}), Proposition \ref{prop:pde} implies the statement ``$\Rightarrow$'' of Theorem \ref{theo:iso}.
\end{cor}

\section{Proof of the statement ``$\Leftarrow$'' of Theorem \ref{theo:iso}}

It is sufficient to show that any potential $V = v_0^2 / 2$ is isonormal to the potential
$$
  \hat V = \hat v^2 / 2, \qquad
  \hat v = \psi_{v_0}^{-1} \circ (2\id),
$$
in the other words, to the potential such that $\psi_{\hat v} = \psi_{v_0}$ and $h_{\hat v} = -\id$.

First we construct $\ph = \ph(v)\in\calF_{even}$ (independent of $\delta$) such that $v(x,\delta)$, the solution of (\ref{pde}) satisfies the equation $v(x,1) = \hat v$. Solution of the initial value problem (\ref{pde}) can be obtained by the method of characteristics. This solution is presented in the following somewhat implicit form:
\begin{equation}
\label{implicit}
  v(x,\delta) = v_0\big(x - \ph\circ v(x,\delta) \delta\big).
\end{equation}
Taking in (\ref{implicit}) $\delta=1$, we obtain
\begin{equation}
\label{hatv}
  \hat v = v_0\circ(\id - \ph\circ\hat v).
\end{equation}
Putting $\ph = \thet\circ\psi_{v_0}$, $h = h_{v_0}$ and using the equations
$$
  \psi_{v_0}\circ v_0 = \id - h, \quad
  \psi_{v_0}\circ\hat v = 2\id,
$$
we obtain from (\ref{hatv}):
$$
  2\id = (\id - h)\circ (\id - \thet\circ (2\id)).
$$
This equation impies
\begin{equation}
\label{thet}
  \thet = \frac12 \id - (\id - h)^{-1}.
\end{equation}

The next step is to prove that the function $\thet$ (equivalently, the function $\ph$) is even. On one hand (\ref{thet}) implies
\begin{equation}
\label{thet-id}
  \thet\circ(-\id) = - \frac12\id - (\id - h)^{-1}\circ (-\id).
\end{equation}
On the other hand, by (\ref{thet})
\begin{eqnarray*}
     \thet
 &=& \frac12\id - (h\circ h - h)^{-1}
  =  \frac12\id - ((h-\id)\circ h)^{-1} \\
 &=& \frac12\id - h\circ (h - \id)^{-1}
  =  \frac12\id - (h - \id)\circ (h - \id)^{-1} - (h - \id)^{-1} \\
 &=& - \frac12 \id - \big( (-\id)\circ(\id - h) \big)^{-1}.
\end{eqnarray*}
Therefore
$$
  \thet = - \frac12 \id - (\id - h)^{-1}\circ (-\id).
$$
Comparing this equation with (\ref{thet-id}), we obtain that $\thet\in\calF_{even}$ and (since $\psi_{v_0}\in\calF_{odd}$) $\ph\in\calF_{even}$.

Hence, the function $F$ (\ref{F}) is well-defined. Equations (\ref{hamF}) determine change of coordinates which transform the Hamiltonian $y^2/2 + v_0^2/2$ to $y^2/2 + \hat v^2/2$. \qed

\section{Involutions}
\label{sec:inv}

\begin{lem}
\label{lem:inv}
Let $h$ be a function from $\calF_*$. The following statements are equivalent.

(A) The function $h$ is an involution: $h\in\calF_{inv}$.

(B) For any $k\in\mN$
$$
  I_{2k+1} := \oint \frac{dz}{(z - h(z))^{2k+1}} = 0.
$$
The integrals are taken along small circles around the origin.

(C) $h = v^{-1}\circ (-\id)\circ v$ for some $v\in\calF_*$.
\end{lem}

{\it Proof}. First we prove ``$(A) \Rightarrow (B)$''. Consider change of the variable $\zeta = h(w)$. Then
(since $h\circ h(w) = w$)
$$
  I_{2k+1} = \oint\frac{h'(w)\, dw}{(h(w) - w)^{2k+1}}
           = - \oint\frac{h'(z)\, dz}{(z - h(z))^{2k+1}}.
$$
Hence $2I_{2k+1}$ equals the integral
$$
  \oint\frac{(z - h(z))'\, dz}{(z - h(z))^{2k+1}} = 0 \quad
  \mbox{for all $k\in\mN$}.
$$

Now we prove ``(B) $\Rightarrow$ (A)''. Consider change of the variable $\zeta = z - h(z)$. Then
\begin{equation}
\label{ph}
  z = \zeta / 2 + \ph(\zeta), \qquad
  \ph = O(\zeta^2)
\end{equation}
and the condition $I_3 = I_5 = \ldots = 0$ takes the form
$$
  \oint\frac{1/2 + \ph'(\zeta)}{\zeta^{2k+1}}\, d\zeta = 0, \qquad k\in\mN.
$$
This means that $\ph\in\calF_{even}$. We obtain the equation
$$
  z - h(z) = \frac\zeta 2 + \ph(\zeta) - h\Big(\frac\zeta 2 + \ph(\zeta)\Big) = \zeta,
$$
so that
\begin{equation}
\label{zeta+ph}
  h(\zeta/2 + \ph(\zeta)) = - \zeta/2 + \ph(\zeta).
\end{equation}
Replace $\zeta$ by $-\zeta$. Since $\ph$ is even, we have:
$$
  h(-\zeta/2 + \ph(\zeta)) = \zeta/2 + \ph(\zeta).
$$
These two equations inply that for any $x$ from a neighborhood of the origin\footnote
{we may take $x = \zeta/2 + \ph(\zeta)$}
$h\circ h(x) = x$.

The implication ``(C) $\Rightarrow$ (A)'' is trivial. To prove ``(B) $\Rightarrow$ (C)'', we define $v\in\calF_*$ by the equation $v^{-1}(\zeta) = \zeta/2 + \ph(\zeta)$, see (\ref{ph}). The function $\ph$ is even. Then by (\ref{zeta+ph})
$h\circ v^{-1} = v^{-1}\circ (-\id)$. This implies $h = v^{-1}\circ (-\id)\circ v$. \qed

{\bf Proof of Proposition \ref{prop:uni}}. For any $v\in\calF$ we prove that $\psi\in\calF_{odd}$ such that $\id - \psi\circ v\in\calF_{inv}$ is unique. Since $(\id - \psi\circ v)'(0) = -1$, we have: $\psi'(0) = 2/v_1$, where as usual $v_1 = v'(0)$. By Lemma \ref{lem:inv} $\psi$ satisfies the equations
\begin{equation}
\label{int1}
   \oint\frac{dz}{(\psi\circ v(z))^{2k+1}} = 0 \quad \mbox{for all } k\in\mN.
\end{equation}
We perform change of the variable
$$
  \zeta = v(z), \quad
  z = w(\zeta), \qquad
  w = v^{-1}, \quad w(0) = 0, \quad w'(0) = 1/v_1.
$$
Equations (\ref{int1}) take the form
\begin{equation}
\label{int2}
   \oint\frac{w'(\zeta)\, d\zeta}{(\psi(\zeta))^{2k+1}} = 0, \qquad k\in\mN.
\end{equation}
We present $\psi$ in the form
$$
  \psi(\zeta) = \frac{\zeta}{\ph(\zeta)}, \qquad
  \ph(0) = \frac1{\psi'(0)} = \frac{v_1}{2}, \quad
  \ph - \ph(0)\in\calF_{even}.
$$
Then conditions (\ref{int2}) can be rewritten in the form:
$$
  \oint\frac{\ph^{2k+1}(\zeta) w'(\zeta)\, d\zeta}{\zeta^{2k+1}} = 0, \qquad k\in\mN
$$
or equivalently, $c_{2k}$, the coefficient at $\zeta^{2k}$ in the Taylor expansion of the function $\ph^{2k+1}(\zeta) w'(\zeta)$ vanishes.

Let
$$
  w(\zeta) = \frac\zeta{v_1} + w_3\zeta^3 + w_5\zeta^5 + \ldots, \quad
  \ph(\zeta) = \frac{v_1}2 + \ph_2\zeta^2 + \ph_4\zeta^4 + \ldots
$$
Then
\begin{equation}
\label{c}
     c_{2k}
  =  \frac{(2k+1)}{v_1} \Big(\frac{v_1}2\Big)^{2k} \ph_{2k} + Q_{2k}(\ph_2,\ph_4,\ldots,\ph_{2k-2}),
\end{equation}
where $Q_{2k}$ is a polynomial with coefficients, depending on ``known quantities'' $v_1,w_3,w_5,\ldots$ Hence the equations $c_{2k}=0$ uniquely determine all the coefficients $\ph_2,\ph_4,\ph_6,\ldots$ from recursive equations (\ref{c}).

We obtain that $\psi$ (and therefore, $h$) are determined by (\ref{uni}) uniquely. \qed
\medskip

{\bf Acknowledgement}. The author thanks Sergey Bolotin for useful discussions and especially for indication to reference \cite{LL} and equation (\ref{quad}).

\end{document}